\documentclass[11pt,oneside]{amsart}

\title{Moduli spaces of twisted K3 surfaces and cubic fourfolds}
\author{Emma Brakkee}
\address{KdVI, University of Amsterdam, P.O.~box 94248,
1090 GE Amsterdam, Netherlands}
\email{e.l.brakkee@uva.nl}

\usepackage[marginratio={1:1, 1:1}]{geometry}
\usepackage{etex,microtype}

\usepackage{amsmath,amssymb}
\usepackage{amsthm,mathrsfs}
\usepackage{geometry}
\usepackage[all, cmtip]{xy}
\usepackage{mathtools}
\usepackage[utf8]{inputenc}
\usepackage{enumitem}
\usepackage{centernot}
\usepackage{url}

\usepackage[hidelinks]{hyperref}

\usepackage[T1]{fontenc}
\usepackage{graphicx}

\newcommand\blfootnote[1]{
  \begingroup
  \renewcommand\thefootnote{}\footnote{#1}
  \addtocounter{footnote}{-1}
  \endgroup
}

\def\dual{\smash{\raisebox{-0.1em}{\scalebox{.7}[1.4]{\rotatebox{90}{\textnormal{\guilsinglleft}}}}}}

\newcommand{\htab}{\hspace*{0.85em}}

\newcommand{\twostar}{$(\ast$$\ast)$} 
\newcommand{\twostarsp}{$(\ast$$\ast)$ } 
\newcommand{\twostarprimesp}{$(\ast$$\ast')$ } 

\newcommand{\labd}{\Lambda_d}
\newcommand{\labddual}{\Lambda^{\dual}_d}
\newcommand{\labdrdual}{\Lambda^{\dual}_{d,r}}
\newcommand{\kdperp}{K_d^{\perp}}
\newcommand{\kdprimeperp}{K_{d'}^{\perp}}
\newcommand{\epart}{E_8(-1)^{\oplus 2}}

\DeclareMathOperator{\tO}{O}
\newcommand{\stO}{\widetilde{\tO}}
\DeclareMathOperator{\Disc}{Disc}

\DeclareMathOperator{\lcm}{lcm}
\DeclareMathOperator{\modulo}{mod}
\renewcommand{\mod}{\:\modulo \:}
\DeclareMathOperator{\ord}{ord}
\DeclareMathOperator{\id}{id}
\DeclareMathOperator{\Pic}{Pic}
\DeclareMathOperator{\ns}{NS}
\DeclareMathOperator{\MathOpHom}{Hom}
\renewcommand{\hom}{\MathOpHom}
\DeclareMathOperator{\MathOpKer}{Ker}
\renewcommand{\ker}{\MathOpKer}

\DeclareMathOperator{\Db}{D^b}

\newcommand{\Z}{\mathbb{Z}}
\newcommand{\Q}{\mathbb{Q}}
\newcommand{\G}{\mathbb{G}}
\newcommand{\C}{\mathbb{C}}

\newcommand{\mM}{\mathcal{M}}
\newcommand{\mN}{\mathcal{N}}
\newcommand{\mO}{\mathcal{O}}
\newcommand{\mF}{\mathcal{F}}
\newcommand{\mD}{\mathcal{D}}
\newcommand{\mC}{\mathcal{C}}
\newcommand{\mP}{\mathcal{P}}
\newcommand{\mQ}{\mathcal{Q}}
\newcommand{\mA}{\mathcal{A}}
\newcommand{\mE}{\mathcal{E}}

\DeclareMathOperator{\tM}{M}
\DeclareMathOperator{\tN}{N}

\DeclareMathOperator{\HH}{H}
\DeclareMathOperator{\spec}{Spec}
\DeclareMathOperator{\br}{Br}
\DeclareMathOperator{\prim}{pr}
\DeclareMathOperator{\lev}{lev}
\DeclareMathOperator{\mar}{mar}
\DeclareMathOperator{\stabmin}{Stab}
\newcommand{\stab}{\stabmin\,}
\DeclareMathOperator{\aut}{Aut}
\DeclareMathOperator{\tors}{tors}
\DeclareMathOperator{\pic}{Pic}
\DeclareMathOperator{\im}{im}

\DeclareMathOperator{\isom}{Isom}

\newcommand{\md}{\tM_d}
\newcommand{\mdfun}{\mM_d}
\newcommand{\mdmar}{\tM_d^{\mar}}
\newcommand{\mdmarfun}{\mM_d^{\mar}}
\newcommand{\mdr}{\tM_d^r}
\newcommand{\mdrfun}{\mM_d^r}
\newcommand{\mdrmar}{\tM_d^{\mar,r}}

\newcommand{\mdlev}{\tM_d^{\lev}}
\newcommand{\mw}{\tM_w}
\newcommand{\tildmw}{\widetilde{\tM}_w}

\newcommand{\relspec}{\underline{\spec}\;}
\newcommand{\sheafhom}{\mathscr{H}\kern -.5pt om}

\theoremstyle{plain}
\newtheorem{theorem}{Theorem}[section]
\newtheorem{proposition}[theorem]{Proposition}
\newtheorem{lemma}[theorem]{Lemma}
\newtheorem{corollary}[theorem]{Corollary}
\newtheorem{maintheorem}{Theorem}

\theoremstyle{definition}
\newtheorem{definition}[theorem]{Definition}
\newtheorem{remark}[theorem]{Remark}

\begin{document}

\begin{abstract}
Motivated by the relation between (twisted) K3 surfaces and special
cubic fourfolds, we construct moduli spaces of polarized twisted K3
surfaces of any fixed degree and order. We do this by mimicking the
construction of the moduli space of untwisted polarized K3 surfaces as a
quotient of a bounded symmetric domain.
\end{abstract}

\maketitle

\blfootnote{The author is supported by the Bonn International Graduate School of Mathematics
and the SFB/TR 45 `Periods, Moduli Spaces and Arithmetic of Algebraic Varieties' of the DFG
(German Research Foundation).}

A twisted K3 surface is a pair $(S,\alpha)$ consisting
of a K3 surface $S$ and a Brauer class $\alpha$ on $S$.
Using the isomorphism $\br(S)\cong\HH^2(S,\mO_X^*)_{\tors}$,
twisted K3 surfaces can be seen as a degree two version of polarized K3 surfaces.
We may also view them from the perspective of Hitchin's generalized K3 surfaces \cite{GenCYstr},
using $\alpha$ to change the volume form on $S$.
This gives us a generalized Calabi--Yau structure,
to which we associate a Hodge structure
$\widetilde{\HH}(S,\alpha,\Z)$ of K3 type on
the full cohomology of $S$ \cite{HuyStelEquivTwisted}.
In this way, we can view $(S,\alpha)$ as a geometric realization of a point in the extended
period domain for K3 surfaces.

\medskip
This paper is concerned with polarized twisted K3 surfaces,
that is, K3 surfaces together with a Brauer class and a primitive ample class in $\HH^2(X,\Z)$.
Our first goal is to construct a moduli space
of these objects, fixing the degree of the polarization and the order
of the Brauer class.
This can be done up to the following concession: when $\rho(S)>1$,
one parametrizes lifts of Brauer classes to $\HH^2(S,\Q)$,
which gives a strictly bigger group than $\br(S)$.
\begin{maintheorem}[see Def.~\ref{DefModFun}, Prop.~\ref{modspace}]
There exists a scheme $\tM_d[r]$ which is a coarse moduli space
for triples $(S,L,\alpha)$ where $S$ is a K3 surface, $L\in\HH^2(S,\Z)$ is a polarization of degree $(L)^2=d$
and $\alpha$ is an element of $\hom(\HH^2(S,\Z)_{\prim},\Z/r\Z)$.
This group has a surjection to $\br(S)[r]$, which is an isomorphism if and only if $\rho(S)=1$.
\end{maintheorem}
We prove this by mimicking the construction of the moduli space of (untwisted) polarized
K3 surfaces via the period domain.
In particular, $\tM_d[r]$ is a quasi-projective variety with at most
finite quotient singularities, whose number of connected components
is at most $r\cdot\gcd(r,d)$ (Proposition \ref{numberofcomps}).

\medskip
In the second part of the paper,
we will concentrate on a Hodge-theoretic relation
between twisted K3 surfaces and special cubic fourfolds.
For untwisted K3 surfaces, this relation was first studied by Hassett \cite{HassettPaper}.
He also constructed,
for $d$ satisfying a numerical condition \twostar, rational maps
\[\md\dashrightarrow \mC_d\]
from the moduli space of polarized K3 surfaces of degree $d$
to the moduli space of special cubic fourfolds of discriminant $d$,
sending a K3 surface to the cubic it is associated to.

\smallskip
Associated twisted K3 surfaces were studied in by Huybrechts in \cite{K3category},
extending results of \cite{AddingtonThomas}.
The numerical condition on the discriminant given by Huybrechts can be formulated as follows:
\[\text{\twostarprimesp}\htab \text{$d'=dr^2$ for some integers $d$ and $r$,
where $d$ satisfies \twostar.}\]
We give a full generalization of Hassett's results
to the setting of twisted K3 surfaces.

\begin{maintheorem}[see Cor.~\ref{geomresult}]
A cubic fourfold $X$ is in $\mC_{d'}$ for some $d'$ satisfying \twostarprimesp if and only if
for every decomposition $d'=dr^2$ with $d$ satisfying \twostar,
$X$ has an associated polarized twisted K3 surface of degree $d$ and order $r$.
\end{maintheorem}

We also give the analogous construction of Hassett's rational maps to $\mC_d$.
Just like for untwisted K3 surfaces, these maps are either birational
or of degree two.
We end with a discussion of the covering involution in the degree two case,
relating this paper to \cite{Involution}.

\subsection*{Acknowledgements}
This work is part of my research as a PhD candidate.
I want to thank my advisor Daniel Huybrechts for suggesting the topic
and for many helpful discussions.
I am also grateful to Thorsten Beckmann, Matthew Dawes and Emmanuel Reinecke
for their help, and to Georg Oberdieck for comments on an earlier version.
Finally, I would like to thank Tony Várilly-Alvarado for his enthusiasm 
and insights.

\subsection{Notation}
\begin{itemize}[leftmargin=12pt]
  \item $\Lambda$ is the lattice isomorphic to the second cohomology
$\HH^2(S,\Z)$ of a K3 surface $S$.
\item $\widetilde{\HH}(S,\Z)$ is the full cohomology of $S$ with the Mukai pairing, viewed as an ungraded ring.
\item $\widetilde{\Lambda}$ is the lattice isomorphic to $\widetilde{\HH}(S,\Z)$. There is an isomorphism $\widetilde{\Lambda}\cong\Lambda\oplus U$.
  \item $\labd\subset\Lambda$ is the orthogonal complement
  of a primitive element $\ell_d\in\Lambda$ of square $d$,
which is unique up to $\tO(\Lambda)$.
\item $\labdrdual:= (\tfrac1r\labddual)/\labddual\cong
\labddual\otimes\Z/r\Z\cong\hom(\labd,\Z/r\Z)$.
\item $\stO(\labd):=\ker(\tO(\labd)\to\tO(\Disc\labd))$.
This group acts naturally on $\labdrdual$.
\item For a lattice isomorphism $\varphi\colon L\to L'$, $\varphi_r$ is the
induced map 
$L^{\dual}\otimes \Z/r\Z\to(L')^{\dual}\otimes\Z/r\Z$.
\item $G[r]$ is the $r$-torsion subgroup of an abelian group $G$. 
\item Cohomology with coefficients in $\G_m$ means étale cohomology.
Otherwise we always use the analytic topology.
\item $\mM_d$ is the moduli functor for polarized K3 surfaces of degree $d$,
and $\Phi\colon\mM_d\to\tM_d$ is the associated coarse moduli space.
\end{itemize}

\begin{remark}
By a moduli functor $\mM$, we will mean a functor on the category of schemes of finite type over $\spec\C$.
A coarse moduli space for $\mM$ is a scheme $\tM$ with a morphism $\xi\colon \mM\to\tM$
such that $\xi(\C)$ is a bijection, and we have factorization over $\tM$ of morphisms $\mM\to T$
for $T$ any $\C$-scheme of finite type.
\end{remark}

\section{Twisted K3 surfaces}

\subsection{Definitions}
For references, see \cite{GlobalTorelliSurvey}, \cite{GeneralizedCY}.
Recall that the \emph{Brauer group} $\br(X)$ of a scheme $X$
is the group of sheaves of Azumaya algebras modulo Morita equivalence,
with multiplication given by the tensor product.
If $X$ is quasi-compact and separated and has an ample line bundle,
then $\br(X)$ is isomorphic to the \emph{cohomological Brauer group}
\[\br(X)':=\HH^2(X,\G_m)_{\tors},\]
which equals $\HH^2(X,\G_m)$ when $X$ is regular and integral.
If $X$ is a complex K3 surface, one can moreover show that
\[\br(X)\cong\HH^2(X,\mO_X^*)_{\tors}\cong(\Q/\Z)^{22-\rho(X)}.\]
A \emph{twisted K3 surface} is a pair $(S,\alpha)$ where $S$ is a K3 surface and
$\alpha\in\br(S)$.
Two twisted K3 surfaces $(S,\alpha)$ and $(S',\alpha')$ are \emph{isomorphic}
if there exists an isomorphism $f\colon S\to S'$ such that $f^*\alpha'=\alpha$.

\medskip
One sees from the exponential sequence on $S$ that any Brauer class $\alpha\in \HH^2(S,\mO_X^*)_{\tors}$
can be written as $\exp(B^{0,2})$ for some $B\in \HH^2(S,\Q)$,
which is unique up to $\HH^2(S,\Z)$ and $\ns(S)\otimes\Q$.
Thus, denoting by $T(S)$ the transcendental lattice of $S$, intersecting with $B$ gives a linear map
${f_{\alpha}=(B,-)\colon T(S)\to\Q/\Z}$ which only depends on $\alpha$.
One can show that $\alpha\mapsto f_{\alpha}$ yields an isomorphism ${\br(S)\cong\hom(T(S),\Q/\Z)}$.

\medskip
Given a lift $B\in \HH^2(S,\Q)$ of $\alpha$,
we define a weight two Hodge structure of K3 type $\widetilde{\HH}(S,B,\Z)$ on the full cohomology of $S$ by 
\[\widetilde{\HH}{}^{2,0}(S,B):=\C[\exp(B)\sigma]\subset \widetilde{\HH}(S,\C),\]
where $\sigma$ is a nowhere degenerate holomorphic 2-form on $S$
and $\exp(B)\sigma := \sigma+B\wedge\sigma$.
This Hodge structure does not depend on our choice of $B$ (up to non-canonical isomorphism \cite[\S2]{HuyStelEquivTwisted}),
so we can define
\[\widetilde{\HH}(S,\alpha,\Z):=\widetilde{\HH}(S,B,\Z)\]
for any $B\in\HH^2(X,\Q)$ with $\exp(B^{0,2})=\alpha$.

\medskip
The \emph{Picard group} of $(S,\alpha)$ is defined as
$\widetilde{\HH}{}^{1,1}(S,\alpha)\cap \widetilde{\HH}(S,\Z)$, so
\[\pic(S,\alpha)=\{\delta\mid (\delta,\exp(B)\sigma)=0\}\subset \widetilde{\HH}(S,\alpha,\Z)\]
for $B\in\HH^2(S,\Q)$ lifting $\alpha$.
If $\alpha$ is trivial, then ${\Pic(S,\alpha)=\HH^0(S,\Z)\oplus\Pic(S)\oplus\HH^4(S,\Z)}$.
The \emph{transcendental lattice} $T(S,\alpha)$ is defined as the orthogonal complement of $\pic(S,\alpha)$
in $\HH^*(S,\alpha,\Z)$. If $\alpha$ is trivial, then $T(S,\alpha)$ is the transcendental lattice $T(S)$ of $S$.
One can show that
\[T(S,\alpha)\cong \ker(f_{\alpha}\colon T(S)\to\Q/\Z)=\{x\in T(S)\mid (B,x)\in\Z\}.\]

\begin{definition}
A \emph{polarized twisted K3 surface} is a triple $(S,L,\alpha)$, where $S$ is a K3 surface,
$L\in\HH^2(S,\Z)$ is a primitive ample class and $\alpha\in\br(S)$.
Two twisted polarized K3 surfaces $(S,L,\alpha)$ and $(S',L',\alpha')$ are \emph{isomorphic}
if there exists an isomorphism $f\colon S\to S'$ such that $f^*L'=L$ and $f^*\alpha'=\alpha$.

We define two invariants of $(S,L,\alpha)$:
its \emph{degree} $d=(L)^2$ and its \emph{order} $r=\ord(\alpha)$
(also known as its \emph{period}).
\end{definition}

\subsection{A non-existence result for moduli spaces}

Ideally, one would like to find a (coarse) moduli space $\tN_d[r]$ for the following functor:
\[\mN_d[r]\colon (Sch/\C)^{o}\to(Sets), T\mapsto \{(f\colon S\to T,L,\alpha)\}.\]
Here, $(f:S\to T,L\in\HH^0(T,R^1f_*\G_m))$ is a smooth proper family
of polarized K3 surfaces of degree $d$ and $\alpha\in\HH^0(T,R^2f_*\G_m)$ 
such that for any closed point $x\in T$, 
base change gives a Brauer class $\alpha_x\in\HH^2(S_x,\G_m)[r]$.

\medskip
It is, however, not difficult to show that $\tN_d[r]$ does not exist
as a locally Noetherian scheme.
Namely, suppose $\mN_d[r]\to \tN_d[r]$ exists.
Consider the natural transformation $\xi\colon \mN_d[r]\to\mM_d$ which at
a scheme $T$ is defined by ${(S\to T,L,\alpha)\mapsto (S\to T,L)}$.
By the properties of a coarse moduli space, there exists a unique morphism $\pi\colon \tN_d[r]\to \tM_d$
which makes the following diagram commute:
\[\xymatrix{
\mN_d[r] \ar_{\xi}[d] \ar[r] & \tN_d[r] \ar@{-->}^{\exists\pi}[d]\\
\mM_d \ar_{\Phi}[r] & \tM_d.
} \]
For a closed point $y\in \tN_d[r]$ corresponding to a tuple $(S,L,\alpha)$,
the image $\pi(y)$ should be the point $x$ of $\tM_d$ corresponding to $(S,L)$.
So the fibre of $\pi$ over $x$ is
\[(\tN_d[r])_x=\{(S,L,\alpha)\mid \alpha\in\br(S),\;\ord\alpha =r\}/\aut(S,L).\]
For $d>2$, let $U\subset \tM_d$ be the open subset where $\aut(S,L)$ is trivial.
Over $U$, we have $(\tN_d[r])_x\cong\br(S)[r]\cong(\Z/r\Z)^{22-\rho(S)}$.
In particular, $\pi|_{\tN_d[r]\times_{\tM_d} U}$ is ramified exactly
over the locus where $\rho(S)>1$.
Now this set is dense in $U$, thus not Zariski closed, giving a contradiction.

For $d=2$, let $U\subset \tM_2$ be the open subset where $\aut(S,L)\cong \Z/2\Z$.
Then over $U$, we have $2^{21-\rho(S)}\leq |(\tN_2[r])_x|\leq 2^{22-\rho(S)}$.
So $\pi|_{\tN_2^r\times_{\tM_2} U}$ is ramified (at least) over the locus where $\rho(S)>2$,
again a dense set in $U$, which leads to a contradiction.

\medskip
When requiring that $\alpha$ has order $r$ (on each connected component of $T$),
non-existence is proven similarly.
One obtains a morphism $\pi$ to $\tM_d$ such that over an open subset $U\subset\tM_d$,
the cardinality of the fibre of $\pi$ over $(S,L)\in U$ is
the number of elements of order $r$ in $(\Z/r\Z)^{22-\rho(S)}$
(or half this number when $d=2$).
Again, $\pi|_{\pi^{-1}(U)}$ is ramified exactly over the locus where $\rho(S)>1$
(at least over the locus where $\rho(S)>2$ when $d=2$), a contradiction.

\section{Moduli spaces of polarized twisted K3 surfaces} 
We will construct a slightly different moduli space $\tM_d[r]$ mapping to $\tM_d$,
whose fibre over $(S,L)\in\tM_d$ parametrizes triples $(S,L,\alpha)$
with $\alpha\in\hom(\HH^2(S,\Z)_{\prim},\Z/r\Z)$.
There is a surjective homomorphism from this group to $\br(S)[r]$,
which is an isomorphism if and only if $\rho(S)=1$.

\subsection{Definition of the moduli functor}
Note that the Kummer sequence \[0\to\mu_r\to\G_m\stackrel{(\cdot)^r}{\to}\G_m\to 0\]
induces a short exact sequence
\[0\to\Pic(S)\otimes\Z/r\Z\to\HH^2(S,\mu_r)\to \br(S)[r]\to 0.\]
If $L\in\HH^2(S,\Z)$ is a polarization, we have injections
$\Z/r\Z\cdot L\hookrightarrow \Pic S\otimes\Z/r\Z\hookrightarrow \HH^2(S,\Z/r\Z)$.
Hence, we get a surjective map
\[\xymatrixrowsep{0.5pc}\xymatrix{
\HH^2(S,\Z/r\Z)/(\Z/r\Z\cdot L) \ar@{=}[d]\ar@{->>}[r] &
\HH^2(S,\Z/r\Z)/(\Pic S\otimes\Z/r\Z) \ar@{=}[d] \;\mathrlap{\cong\br(S)[r]}\\
\HH^2(S,\Z)^{\dual}_{\prim}\otimes\Z/r\Z & T(S)^{\dual}\otimes\Z/r\Z &
}\]
which is an isomorphism if and only if $\rho(S)=1$.

\medskip
We define a relative version of
$\HH^2(S,\Z)^{\dual}_{\prim}\otimes\Z/r\Z\cong\hom(\HH^2(S,\Z)_{\prim},\Z/r\Z)$ as follows.
For a smooth proper family $(f\colon S\to T,L)$ of polarized K3 surfaces, set
\[R^2_{\prim}f_*\Z := \ker\Big(R^2f_*\Z\xrightarrow{\cdot c_1(L)} R^4f_*\Z\Big)\]
where $c_1(L)$ is the image of $L$ under $\HH^0(T,R^1f_*\G_m)\to \HH^0(T,R^2f_*\Z)$.
Let $\mF[r]$ be the following local system:
\[\mF[r]:=\sheafhom(R^2_{\prim}f_*\Z,\underline{\Z/r\Z}).\]

\begin{definition}\label{DefModFun} The moduli functor $\mM_d[r]$ is defined as
\[\mM_d[r]\colon (Sch/\C)^{o}\to(Sets),\; T\mapsto\{(f\colon S\to T,L,\alpha)\}/_{\cong}\]
where $\big(f:S\to T,L\in\HH^0(T,R^1f_*\G_m)\big)$
is a smooth proper family of polarized K3 surfaces of degree $d$
and $\alpha\in\HH^0(T,\mF[r])$.
We define
\[\mM_d^r \colon (Sch/\C)^{o}\to(Sets)\]
to be the subfunctor
sending a scheme $T$ to the set of those tuples $(f,L,\alpha)$
for which $\alpha$ has order $r$ on each connected component of $T$.
\end{definition}
We will construct coarse moduli spaces for $\mM_d[r]$ and $\mM_d^r$.

\subsection{Construction of the moduli space}
Recall the construction of $\tM_d$ as a subvariety of a quotient of a bounded symmetric domain
(see e.g.\ \cite{LecturesOnK3}).
The moduli functor $\mM_d^{\mar}$ of marked polarized K3 surfaces of degree $d$ is given by
\[\mM_d^{\mar}(T)=\{(f\colon S\to T,L\in\HH^0(T,R^1f_*\G_m),
\varphi\colon R^2_{\prim}f_*\Z\cong\underline{\Lambda_d})\}.\] 
It has an analytic fine moduli space $\tM_d^{\mar}$, 
which can be constructed as an open submanifold of the period domain $\mD(\labd)$ of $\labd$.
In particular, there exists a universal family
\[\left(f\colon S^{\mar}\to \tM_d^{\mar},L^{\mar},\varphi^{\mar}\right).\]
We denote the morphism $\mdmarfun\to \mdmar$ by $\Phi^{\mar}$.
The moduli space $\tM_d$ is obtained from $\tM_d^{\mar}$ by taking the quotient under
the action of $\stO(\Lambda_d)$.

\medskip
Note that $\varphi^{\mar}$ induces an isomorphism
$\varphi^{\mar}_r\colon\mF[r]\cong\underline{\hom(\labddual,\Z/r\Z)}\cong\underline{\labdrdual}$,
thus, $\mF[r]$ is the sheaf of sections of the trivial finite cover
\begin{align*}
\mdmar[r]&:=
\relspec \sheafhom(\underline{\labdrdual},\mO_{\tM_d^{\mar}})\\
&\;=\tM_d^{\mar}\times \labdrdual.
\end{align*}
The space $\mdmar[r]$ is a coarse moduli space for the functor
\[\mdmarfun[r]\colon (Sch/\C)^{o}\to(Sets),\; T\mapsto\{(f\colon S\to T,L,\varphi,\alpha)\}\]
where $(f\colon S\to T,L,\varphi)\in\mdmar(T)$ and $\alpha\in\HH^0(T,\mF[r])$,
and the morphism 
\[\Phi^{\mar}[r]\colon \mdmarfun[r]\to\mdmar[r]\] 
is defined over a scheme $T$ by
\[(S\to T,L,\varphi,\alpha)\mapsto (\Phi^{\mar}(S\to T,L,\varphi),\varphi_r(\alpha)).\]
So we have a commutative diagram
\[\xymatrixcolsep{3.5pc}\xymatrix{ \mdmarfun \ar[r]^{\Phi^{\mar}} & \mdmar \\
 \mdmarfun[r] \ar[u] \ar[r]_{\Phi^{\mar}[r]} & \mdmar[r] \ar[u]
} \]

The action of $\stO(\Lambda_d)$ on $\tM_d^{\mar}$ lifts to $\mdmar[r]$ via
\[g(S,L,\varphi,\alpha) = \left(S,L,g\circ\varphi,\varphi_r^{-1}g\varphi_r(\alpha)\right).\]
Under $\mdmar[r]\hookrightarrow\mD(\labd)\times\labdrdual$,
this is the restriction of the natural action of
$\stO(\labd)$ on $\mD(\labd)\times\labdrdual$.
This action is properly discontinuous: it is on $\mD(\labd)$
(see \cite[Rem.~6.1.10]{LecturesOnK3}),
so also on the product with the finite group $\labdrdual$.
It follows that the quotient
\[\md[r]:=\mdmar[r]/\stO(\Lambda_d)\] 
exists as a complex space.
Similarly, let $\mdrmar\subset \mdmar[r]$ be the union of those components $\mdmar[r]\times\{v\}$ 
for elements $v\in\labdrdual$ of order $r$. Then the quotient
\[\mdr:=\mdrmar/\stO(\Lambda_d)\]
exists as a complex space.

\medskip
We claim that $\md[r]$ and $\mdr$ are quasi-projective varieties.
Consider the following commutative diagram:
\[\xymatrix{\mdmar[r]=\tM_d^{\mar}\times \labdrdual \ar[d] \ar^-{\pi}[r]
& \labdrdual\ar[d] \\
\tM_d[r] \ar^-{\bar{\pi}}[r] & \labdrdual/\stO(\Lambda_d).
}\]
Giving the sets on the right side the discrete topology, all these maps are continuous.
So under $\bar{\pi}$, each connected component of $\tM_d[r]$
is mapped to a point.
Vice versa, given $[w]\in\labdrdual$, the inverse image 
of $\stO(\labd)\cdot[w]\in\labdrdual/\stO(\labd)$ under $\bar{\pi}$ is
\[\mw:=\left(\mdmar\times\, \stO(\Lambda_d)\cdot [w]\right)/\stO(\Lambda_d)
\cong \big(\mdmar\times\{[w]\}\big)/\stab[w]\]
where $\stab[w]\subset \stO(\Lambda_d)$ is the stabilizer of $[w]$
under the acion of $\stO(\Lambda_d)$ on $\labdrdual$.
Now $\stab [w]$ contains the reflection $s_{\delta}$ for an element $\delta\in\labd$
of square $-2$ orthogonal to $w$ and $\ell_d'$, which interchanges
the two connected components of $\mdmar$.
Hence, $\mw$ is connected; even irreducible.
This shows that the connected components of $\tM_d[r]$ are in one-to-one correspondence with
$\labdrdual/\stO(\Lambda_d)$.
Each component $\mw$ parametrizes triples $(S,L,\alpha)$
that admit a marking $\varphi$ with $\varphi_r(\alpha)=[w]$. 
The components belonging to $\tM_d^r$ are those $\mw$ for which $[w]$ has order $r$.

\begin{corollary}\label{ConnComp}
 Every connected component of $\md[r]$ (and therefore of $\mdr$) is an irreducible,
 quasi-projective variety with at most finite quotient singularities.
\end{corollary}
\begin{proof}
The finite index subgroup $\stab[w]\subset \stO(\Lambda_d)$ being arithmetic,
the quotient space $\mD(\labd)/\stab[w]$ is a quasi-projective variety with finite
quotient singularities, by \cite{BB} and \cite[Lemma~IV.7.2]{Satake}.
We will show that $\mw=\big(\mdmar\times\{[w]\}\big)/\stab[w]$
is a Zariski open subset of it,
using the same argument as for the algebraicity of the moduli space
of untwisted polarized K3 surfaces
(see e.g.\ \cite[\S6.4.1]{LecturesOnK3}).

Let $\ell$ be a large enough multiple of $r$ such that there exists a fine moduli space
$\mdlev$ of polarized K3 surfaces with a $\Lambda/\ell\Lambda$-level structure,
see Remark \ref{levmod} below.
For the universal family $\pi\colon S^{\lev}\to\mdlev$, there exists a marking
$R^2_{\prim}\pi_*\Z\otimes\underline{\Z/\ell\Z}\cong \underline{\labd\otimes\Z/\ell\Z}$.
This induces a holomorphic map $\mdlev\to\mD(\labd)/\Gamma_{\ell}$, where
\[\Gamma_{\ell}=\{g\in\stO(\labd)\mid g\equiv \id\mod \ell\}\subset\stab[w].\]
The image of this map is $\mdmar/\Gamma_{\ell}$.
Dividing out further by $\stab[w]$ yields a holomorphic map
\[\mdlev\twoheadrightarrow \mdmar/\stab[w]\subset\mD(\labd)/\stab[w].\]
By a theorem of Borel \cite{Borel} (and also \cite[Lemma~IV.7.2]{Satake}),
this map is algebraic, and therefore the image $\mdmar/\stab[w]$ is constructible.
Since it is also analytically open in $\mD(\labd)/\stab[w]$,
it is Zariski open \cite[Cor.~XII.2.3]{SGA1}.
\end{proof}

\begin{remark}\label{levmod}
Recall (see e.g.\ \cite[\S6.4.2]{LecturesOnK3}) that for $\ell$ large enough, 
there exists a fine moduli space $\mdlev$
of polarized K3 surfaces $(S,L)$ of degree $d$ with a \emph{$\Lambda/\ell\Lambda$-level structure},
i.e.\ an isometry $\HH^2(S,\Z)_{\prim}\otimes\Z/\ell\Z\cong \labd\otimes\Z/\ell\Z$.
The space $\mdlev$ is a smooth quasi-projective variety which is a finite cover of $\md$.
We could have constructed $\tM_d[r]$ as a quotient of 
$\mdlev\times\labdrdual$ instead,
choosing $\ell$ to be a large enough multiple of $r$.
\end{remark}

One constructs a morphism $\Psi\colon\mM_d[r]\to \tM_d[r]$ in the following way.
Consider a point $(f\colon S\to T,L,\alpha)$ in $\mM_d[r](T)$.
Proceeding as for untwisted polarized K3 surfaces, we pass to the (infinite) étale covering
\[\widetilde{T}:=\isom(R_{\prim}^2f_*\Z,\underline{\labd})
\overset{\eta}{\longrightarrow}T,\]
which has a natural $\stO(\labd)$-action, satisfying
$\widetilde{T}/\stO(\labd)\cong T$.
Write $\widetilde{f}\colon \widetilde{S}\to\widetilde{T}$ for the pullback family.
The local system $R^2_{\prim}\widetilde{f}_*\Z$ is trivial: there
exists a canonical isomorphism 
${\varphi\colon R^2_{\prim}\widetilde{f}_*\Z\cong\underline{\Lambda_d}}$.
Now $\Phi^{\mar}[r]\big(\widetilde{S},\eta^*L,\varphi,\eta^*\alpha\big)$
is an element of $\mdmar[r](\widetilde{T})$,
i.e.\ a holomorphic map $\widetilde{T}\to \mdmar[r]$.
This map is $\stO(\labd)$-equivariant, hence descends to a map $T\to \md[r]$.
This map is algebraic by \cite{Borel}, thus defines a point in $\tM_d[r](T)$.
We let $\Psi(S\to T,L,\alpha)$ be this point.

\begin{proposition}\label{modspace} 
The space $\tM_d[r]$ is a coarse moduli space for the functor $\mM_d[r]$.
\end{proposition}
\begin{proof}
By definition, there is a commutative diagram
\[\xymatrixcolsep{3.5pc}\xymatrix{ \mdmarfun[r] \ar[r]^{\Phi^{\mar}[r]} \ar[d]_{F} & \mdmar[r] \ar[d]^{q}\\
\mdfun[r]\ar[r]_{\Psi} & \md[r]
}\]
where the map $F$ forgets the marking and $q$ is the quotient map.
We need to show that $\Psi(\C)\colon \mdfun[r](\C)\to\md[r](\C)$ is a bijection.
For $x\in\md[r](\C)$, let $y\in\mdmar[r](\C)$ such that $q(y)=x$.
Set $\Psi(\C)^{-1}(x):= F(\Phi^{\mar}[r](\C)^{-1}(y))$;
note that this does not depend on the choice of $y$.
One checks that $\Psi(\C)^{-1}$ defines a set-theoretic inverse to $\Psi(\C)$.

\smallskip
For the universal property of $\Psi$, let $s\colon \mdfun[r]\to T$ 
be a morphism to a finite type $\C$-scheme $T$.
Then $s\circ F$ is a map from $\mdmarfun[r]$ to $T$;
since $\mdmarfun[r]\to\mdmar[r]$ is a coarse moduli space,
this induces a unique holomorphic map $t\colon \mdmar[r]\to T$ such that 
$t\circ \Phi^{\mar}[r]=s\circ F$.
It follows from the uniqueness that $t$ is equivariant, thus factors over a holomorphic map
$\md[r]\to T$. We will show that this map is algebraic.

Like before, let $\ell$ be a large enough multiple of $r$
such that there exists a fine moduli space $\mdlev$
of K3 surfaces with a $\Lambda/\ell\Lambda$-level structure.
The map $\mdmar[r]\to T$ factors as 
\[\mdmar[r]\to \mdlev\times\labdrdual\to\md[r]\to T\]
(see Remark \ref{levmod}).
The map $\mdlev\times\labdrdual\to T$ is algebraic and equivariant
under the algebraic action of $\stO(\Lambda_d)$.
The induced algebraic morphism $(\mdlev\times\labdrdual)/\stO(\Lambda_d)\to T$
is the given map $\md[r]\to T$.
\end{proof}

The proof that $\mdr$ is a coarse moduli space for $\mdrfun$ is 
analogous.

\begin{proposition} \label{numberofcomps}
The moduli space $\mdr$ has at most $r\cdot\gcd(r,d)$ many connected components.
\end{proposition}
This follows directly from the following lemma.
Denote $\Lambda=\epart\oplus U_1\oplus U_2\oplus U_3$.
Let $\{e_i,f_i\}$ be the standard basis for the $i$-th copy of $U$.
Fix $\ell_d:=e_3+\tfrac{d}{2}f_3$ and $\ell_d':= e_3-\tfrac{d}{2}f_3$,
so $\labddual\cong\epart\oplus U_1\oplus U_2\oplus\langle\tfrac{1}{d}\ell_d'\rangle$. For integers $n$, $k$, we let
\[w_{n,k}:=\tfrac1r(e_1+nf_1+\tfrac{k}{d}\ell_d')\in\tfrac1r\labddual.\]
\begin{lemma}\label{orbits}
Every element of order $r$ in $\labdrdual$ is equivalent
under the action of $\stO(\Lambda_d)$
to $[w_{n,k}]$ for some $n,k\in\Z$.
Moreover, if $n\equiv n'\mod r$ and $k\equiv k'\mod\gcd(r,d)$,
then $[w_{n,k}]$ and $[w_{n',k'}]$ are equivalent.
\end{lemma}
\begin{proof}
Elements in $\labdrdual$ of order $r$ are of the form
$m[\tfrac1r x]$ where $\gcd(m,r)=1$ and $x\in \labddual$ is primitive, so
$x=sy+\tfrac{t}{d}\ell_d'$ 
for some primitive $y\in\epart\oplus U_1\oplus U_2$ and integers $s$, $t$
with $\gcd(s,t)=1$.
Write $d=d_0\cdot\gcd(d,t)$ and $t=t_0\cdot\gcd(d,t)$. Then
$d_0x=d_0sy+t_0\ell_d'\in\labd$
is primitive and
\begin{align*}
(d_0W,\labd)&=\gcd\left( (d_0sy,\epart\oplus U_1\oplus U_2),(t_0\ell_d',\Z\ell_d')\right)\\
&=\gcd(d_0s,dt_0)\\
&=d_0.
\end{align*}
By Eichler's criterion \cite[Prop.~3.3]{GrHuSa},
$d_0x$ is equivalent under $\stO(\labd)$ to
$d_0(e_1+nf_1)+t_0\ell_d'$
for some $n$. So $\tfrac1r x$ is equivalent to $\tfrac1r(e_1+nf_1+\tfrac{t}{d}\ell_d')=w_{n,t}$.

Now $\tfrac{m}{r}x\equiv mw_{n,t}$ is equivalent modulo $\labddual$ to
$\tfrac1r(me_1+(mn+r)f_1+\tfrac{mt}{d}\ell_d')$. 
As $\gcd(r,m)=1$, the element 
$y=me_1+(mn+r)f_1+\tfrac{mt}{d}\ell_d'\in\labddual$ is primitive,
so by the above, $\tfrac1r y$ is equivalent under $\stO(\labd)$ to some $w_{n',t'}$. It follows that $m[\tfrac1r x]\in\labdrdual$ is equivalent to $[w_{n',t'}]$.

\medskip
Now note that if $t'\equiv t\mod d$, then $w_{n,t}$ is equivalent to
$w_{n',t'}$ for some $n'$ (by Eichler's criterion).
In particular, writing $\gcd(r,d)=pr+qd$, the class
\[[w_{n,\gcd(r,d)+t}]=[\tfrac1r(e_1+nf_1+(pr+qd+t)\tfrac{1}{d}\ell_d')]
=[\tfrac1r(e_1+nf_1+(qd+t)\tfrac{1}{d}\ell_d')]\]
in $\labdrdual$ is equivalent to $[\tfrac1r(e_1+n'f_1+\tfrac{t}{d}\ell_d')]=[w_{n',t}]$
for some $n'$.
This shows that every $[w_{n,k}]$ is equivalent to some $[w_{n',k'}]$ with
$0\leq n'< r$ and $0\leq k'<\gcd(r,d)$.
\end{proof}

\section{Twisted K3 surfaces and cubic fourfolds}
Recall that a smooth cubic fourfold $X$ is \emph{special}
if the lattice
$\HH^{2,2}(X)\cap \HH^4(X,\Z)$
has rank at least two. Special cubic fourfolds form a countably infinite union
of irreducible divisors $\mC_d$ in the moduli space of cubic fourfolds. 
Here $\mC_d\neq\emptyset$ if and only if $d>6$ and $d\equiv 0,2\mod 6$.
Hassett \cite{HassettPaper} showed that $X$ is in $\mC_d$ with $d$ satisfying
\[\text{\twostarsp}\htab \text{$d$ is even and not divisible by 4, 9,
or any odd prime $p\equiv 2$ mod 3}\]
if and only if there exists a polarized K3 surface $(S,L)$ of degree $d$
whose primitive cohomology $\HH^2(S,\Z)_{\prim}$ can be embedded
Hodge-isometrically into $\HH^4(X,\Z)$, up to a sign and a Tate twist.

\medskip
We will generalize this result to twisted K3 surfaces in
Theorem~\ref{LatPrecise} and Corollary~\ref{geomresult}.
An important ingredient will be the period map for polarized twisted K3 surfaces.

\subsection{Period maps for twisted K3 surfaces}\label{Periodmaps}

We have seen that the connected components of $\tM_d^r$ are of the form
\[\mw=(\mdmar\times\{[w]\})/\stab[w]\]
for $[w]\in\labdrdual$ of order $r$.
We will construct a period map from
$\mdmar\times\{[w]\}$ to the period domain $\mD(T_w)$ of the lattice
\[T_w:=\ker((w,-)\colon \labd\to\Q/\Z).\]

Let $(S,L,\varphi,[w])\in \mdmar\times\{[w]\}$. The corresponding twisted Hodge structure
$\widetilde{\HH}(S,[w],\Z)$ on $S$ is given as follows.
Let
\[w'=(\varphi^{\dual})^{-1}(w)\in\tfrac1r\HH^2(S,\Z)^{\dual}_{\prim}\subset\HH^2(S,\Q).\]
Then
$\widetilde{\HH}{}^{2,0}(S,[w])$ is $\C[\sigma + w'\wedge\sigma]$,
where $\sigma$ is a non-degenerate holomorphic 2-form on $S$. 
Let $\widetilde{\Lambda}=\Lambda\oplus U_4$ be the extended K3 lattice.
We can extend $\varphi$ to an isomorphism 
$\widetilde{\varphi}\colon \widetilde{\HH}(S,\Z)\to\widetilde{\Lambda}$
by sending $1\in\HH^0(S,\Z)$ to $e_4\in U_4$ and $1\in\HH^4(S,\Z)$ to $f_4\in U_4$.
Then
\[\widetilde{\varphi}(\sigma + w'\wedge\sigma)=
\varphi(\sigma)+(w,\varphi(\sigma))f_4.\]

Recall that for an even lattice $N$ and $B\in N$, 
the \emph{$B$-field shift} ${\exp(B)\in \tO(N\oplus U)}$ is defined by
\[z\mapsto z-(B.z)f,\:\: e\mapsto e+B-\tfrac{(B)^2}2f,\:\:f\mapsto f\]
for $z\in N$, 
where $\{e,f\}$ is the standard basis of the hyperbolic plane $U$.
For $B\in N_{\Q}$, we define $\exp(B)\in\tO((N\oplus U)_{\Q})$
by linear extension.
The discussion above shows that $\widetilde{\varphi}(\sigma + w'\wedge\sigma)=
\exp(w)\varphi(\sigma)$.
We thus obtain a map
\[\mQ_w\colon\mdmar\times\{[w]\}\to\mD\big((\exp(w)\labd)\cap\widetilde{\Lambda}\big)\]
sending $(S,L,\varphi,[w])$ to $[\widetilde{\varphi}(\HH^{2,0}(S,[w]))]$.

\medskip
The above depends on the choice of a representative $w\in\tfrac1r\labddual$
of $[w]\in\labdrdual$.
We can get rid of this choice in the following way.
First, the lattice $T_w$ is a finite index sublattice of $\labd$, 
so we have $\mD(T_w)=\mD(\labd)$.
Second, note that the map $\exp(w)$ gives an isomorphism
$T_w\cong (\exp(w)\labd)\cap\widetilde{\Lambda}$.
We see that $\mQ_w$ factors over the usual period map $\mP$ for $\mdmar$:
the diagram
\[\xymatrix{ \mdmar\times\{[w]\}\cong\mdmar \ar[d]_{\mQ_w}\ar[r]^-{\mP} & \mD(\labd)\ar[d]^{=}\\
\mD\big((\exp(w)\labd)\cap\widetilde{\Lambda}\big) \ar[r]_-{\cong} & \mD(T_w)
}\]
commutes.
Denote by $\mP_w$ the composition from $\mdmar\times\{[w]\}$ to $\mD(T_w)$.
It follows from the above diagram that $\mP_w$ is holomorphic and injective.

\medskip
Just like $\mw$, the quotient $\mD(T_w)/\stab[w]$ is a quasi-projective variety by \cite{BB}.
There is a commutative diagram
\[\xymatrix{\mdmar\times\{[w]\} \ar[d] \ar@{^{(}->}[r]^-{\mP_w} & \mD(T_w) \ar[d] \\
\mw \ar@{^{(}->}[r]^-{\overline{\mP}_w} & \mD(T_w)/\stab[w]
}\]
where $\overline{\mP}_w$ is algebraic by the same argument as in Corollary~\ref{ConnComp}
(note that when $\ell$ is a multiple of $r^2d$,
the group $\Gamma_{\ell}=\{g\in\stO(\labd)\mid g\equiv \id\mod \ell\}$
is contained in $\stab[w]$).

\medskip
Recall (see e.g.\ \cite[Rem.~6.4.5]{LecturesOnK3}) that
$\mD(T_w)\backslash\im\mP_w=\mD(\labd)\backslash\im \mP$ is a union of hyperplanes
$\bigcup_{\delta\in\Delta(\labd)}\delta^{\perp}$, where $\Delta(\labd)$
is the set of $(-2)$-classes in $\labd$.
It follows that $\mD(T_w)$ parametrizes periods of twisted \emph{quasi-polarized} K3 surfaces,
i.e.\ twisted K3 surfaces with a line bundle that is nef and big
(however, the corresponding moduli stack is not separated).
Hence, the quotient $\mD(T_w)/\stab[w]$ can be viewed as a moduli space of
quasi-polarized twisted K3 surfaces.

\subsection{The lattice \texorpdfstring{$T_w$}{Tw}}\label{Tw}
We denote by $\HH^4(X,\Z)(1)$ the middle cohomology of a cubic fourfold $X$
with the intersection product changed by a sign
and the weight of the Hodge structure shifted by two.
The cubic $X$ lies in the divisor $\mC_d$ if and only if there exists
a primitive sublattice 
\[K_d\subset\HH^{2,2}(X)\cap\HH^4(X,\Z)(1)\]
of rank two and discriminant $d$ containing the square of the hyperplane class.
The orthogonal complement $\kdperp\subset \HH^4(X,\Z)(1)$ has an induced Hodge structure
which determines $X$ 
when $X\in\mC_d$ is very general.
As a lattice, $\kdperp$ only depends on $d$.
Theorem~\ref{LatPrecise} below
tells us that for certain $d$ and $r$, and $w\in\hom(\Lambda_d, \Q/\Z)$ of order $r$,
the lattice $T_w$ is isomorphic to $K_{dr^2}^\perp$.

\medskip
We will see that the discriminant group of $T_w$
can always be generated by three elements. 
As $T_w$ has signature $(2,19)$, it follows that $T_w$
is determined by its discriminant group
and the quadratic form on it \cite[Cor.~1.13.3]{Nikulin}.
To prove Theorem~\ref{LatPrecise}, it thus suffices to determine when
\[\bigl(\Disc T_w,q_{T_w}\bigr)\cong\bigl(\Disc K_{dr^2}^\perp,q_{K_{dr^2}^\perp}\bigr).\]
Write $d'=dr^2$.
We will use the following result by Hassett (using our sign convention):
\begin{proposition}\textnormal{\cite[Prop.~3.2.5]{HassettPaper}}\label{HassettQuadForm}
When $d'\equiv 0\mod 6$, then $\Disc(\kdprimeperp)$ is isomorphic to
$\Z/3\Z\times\Z/\tfrac{d'}3\Z$,
which is cyclic unless $9$ divides $d'$.
One can choose generators $(1,0)$ and $(0,1)$ such that the quadratic form $q_{\kdprimeperp}$
satisfies $q_{\kdprimeperp}(1,0)=3/d'\mod 2\Z$
and ${q_{\kdprimeperp}(1,0)=-2/3\mod 2\Z}$.

When $d'\equiv 2\mod 6$, then $\Disc(\kdprimeperp)$ is $\Z/d'\Z$.
One can choose a generator $u$ such that
$q_{\kdprimeperp}(u) = (1-2d')/3d'\mod 2\Z$.
\end{proposition}

\subsubsection{The discriminant group of \texorpdfstring{$T_w$}{Tw}}

Let $w\in\tfrac1r\labddual$ such that $[w]\in\labdrdual$ has order $r$.
We will describe the group $\Disc T_w$ and the quadratic form on it.
Note that if $g\in\stO(\labd)$, then
$g$ induces an isomorphism $T_w\cong T_{g(w)}$.
So by Lemma~\ref{orbits}, we can assume that
\[w=w_{n,k}=\tfrac1r(e_1+nf_1+\tfrac{k}{d}\ell_d')\]
for some $n$, $k$.
Then $T_w=\epart\oplus U_2\oplus T_0$, where
\[T_0=\{y\in U_1\oplus\Z\ell_d'\mid (y,w)\in \Z\}=\langle e_1-nf_1,rf_1,kf_1+\ell_d'\rangle.\]
Since $\epart\oplus U_2$ is unimodular, $\Disc T_w$ is isomorphic to $\Disc T_0$.
The intersection matrix of $T_0$ is (compare \cite[Lemma~2.12]{Reinecke})
\[M=\begin{pmatrix}
-2n & r & k\\
r & 0 & 0\\
k & 0 & -d
\end{pmatrix}\]
As the map $T_0\to T_0^{\dual}$ is given by the matrix $M^t=M$, we have
\[\Disc T_0=\Z/g_1\Z\times\Z/g_2\Z\times\Z/g_3\Z\]
where the invariant factors $g_i$ can be computed using the $i\times i$-minors
of $M$ \cite[Satz~2.9.6]{Bosch}:
\[g_1=\gcd(2n,r,k,d),\; g_2=\gcd(r^2,kr,rd,2nd-k^2)/g_1,\; g_3=dr^2/g_1g_2.\]

We obtain
\begin{proposition}\label{cyclic}
Let $w=w_{n,k}\in\tfrac1r\labddual$.
\begin{enumerate}[label=\emph{(\roman*)}]
\item The group $\Disc T_w$ is cyclic if and only if $\gcd(r,2nd-k^2)=1$.
\item We have
\[\Disc T_w\cong \Z/(r^2d/3)\Z\times\Z/3\Z\]
if and only if $\gcd(r,2nd-k^2)=3$, and if $3|d$ then $9\centernot|nd$.
\end{enumerate}
\end{proposition}

In order to determine the quadratic form on $\Disc T_w$,
we write down explicit generators.
Consider the following elements of $\Disc T_0$:
\begin{align*}
 [f_1] &= [\tfrac{1}{r}(rf_1)]\\
[\ell_d'/d]&=[\tfrac{1}{d}(kf_1+\ell_d')-\tfrac{k}{rd}(rf_1)]\\
[w]&=[\tfrac{1}{r}(e_1-nf_1)+\tfrac{2nd-k^2}{r^2d}(rf_1)+\tfrac{k}{rd}(kf_1+\ell_d')].
\end{align*}
Note that
\[\ord [w] = \lcm\left(\tfrac{r^2d}{\gcd(r^2d,2nd-k^2)},\tfrac{rd}{\gcd(rd,k)},r\right).\]
One checks that the group generated by $[f_1]$, $[\ell_d'/d]$ and $[w]$ has
$r^2d$ many elements,
so it is the whole discriminant group:
\[\Disc T_w=\langle [f_1],[\ell_d'/d],[w]\rangle.\]

\begin{lemma}\label{choosegcd}
 If $\gcd(d,k,r)=s$, then there is an integer $p$ such that $\gcd(d,k+pr)=s$.
\end{lemma}
\begin{proof}
 Let $d=d_0\gcd(r,d)$, so $\gcd(r,sd_0)=s$. Write $xsd_0+yr=s$. Then
 \begin{align*}
  \gcd\big(d,k+(1-\tfrac{k}{s})yr\big) 
  &= \gcd\big(sd_0,k+(1-\tfrac{k}{s})(s-xsd_0)\big)\\
  &=s\gcd\left(d_0,1-xd_0+xd_0\tfrac{k}{s}\right)\\
  &=s. \qedhere
 \end{align*}
\end{proof}

First assume $\Disc T_w$ is cyclic, so $\gcd(r,2nd-k^2)=1$.
In particular, this implies ${\gcd(r,d,k)=1}$.
By Lemma~\ref{choosegcd} there exists a $p$ such that $\gcd(d,k+pr)=1$.
Since $T_{w_{n,k}}\cong T_{w_{n,k+pr}}$, we can replace $k$ by $k+pr$.
Then we have ${\gcd(r^2d,2nd-k^2)=1}$;
 hence, $[w]$ generates $\Disc T_w$.
 So the quadratic form $q_{T_w}$ on $\Disc T_w$ is determined by
\[q_{T_w}([w])=[(w)^2]=\tfrac{1}{r^2d}(2nd-k^2)\mod 2\Z.\]

\medskip
Next, assume $\Disc T_w\cong \Z/(r^2d/3)\Z\times\Z/3\Z$.
If $3$ divides $d$, we have ${\gcd(r,2nd-k^2)=3}$, and $9\centernot| nd$ implies 
$9\centernot|2nd-k^2$.
It follows that $\gcd(r^2,2nd-k^2)=3$, so $[w]$ generates $\Z/(dr^2/3)\Z$.
As a generator of the factor $\Z/3\Z$, we take the element
\[u:=\tfrac{k}{3}[f_1]-\tfrac{d}{3}[\ell_d'/d]=\tfrac{1}{3}[kf_1+\ell_d'].\]
We have $(u)^2=-\tfrac{d}{9}\mod 2\Z$.

If $3$ does not divide $d$, we may have $9|2nd-k^2$, but this implies $9\centernot|r$.
Using $T_{w_{n,k}}\cong T_{w_{n+r,k}}$,
we may replace $n$ by $n+r$ and obtain $9\centernot|2nd-k^2$.
This gives $\gcd(r^2,2nd-k^2)=3$, so $[w]$ generates $\Z/(dr^2/3)\Z$.
As a generator of the factor $\Z/3\Z$, we take
\[u':=\tfrac{rd}{3}[w]-\tfrac{2nd-k^2}{3}[f_1]=\tfrac{1}{3}([d(e_1-nf_1)+k(kf_1+\ell_d')]).\]
We have $(u')^2=-\tfrac{d}{9}(2nd-k^2)\mod 2\Z$.

\begin{corollary}\label{generatorsDisc}
Let $w=w_{n,k}\in\tfrac1r\labddual$.
\begin{enumerate}[label=\emph{(\roman*)}]
\item If $\Disc T_w$ is cyclic, there exists a generator $t$ such that
\[q_{T_w}(t)=\tfrac{1}{r^2d}(2nd-k^2)\mod 2\Z.\]
\item If $\Disc T_w\cong \Z/(r^2d/3)\Z\times\Z/3\Z$ and $3\centernot| d$,
there exist generators $(1,0)$ and $(0,1)$ such that
\[q_{T_w}(1,0)=\tfrac{1}{r^2d}(2nd-k^2)\mod 2\Z\]
and
\[q_{T_w}(0,1)=-\tfrac{d}{9}(2nd-k^2)\mod 2\Z.\]
\item If $\Disc T_w\cong \Z/(r^2d/3)\Z\times\Z/3\Z$ and $3|d$,
there exist generators $(1,0)$ and $(0,1)$ such that
\[q_{T_w}(1,0)=\tfrac{1}{r^2d}(2nd-k^2)\mod 2\Z\]
and
\[q_{T_w}(0,1)=-\tfrac{d}{9}\mod 2\Z.\]
\end{enumerate}
\end{corollary}

\subsection{Existence of associated twisted K3 surfaces}
Consider the following condition on $d'\in\Z$.
\[\text{\twostarprimesp}\htab \text{$d'=dr^2$ for some integers $d$ and $r$,
where $d$ satisfies \twostar.}\]
The results of Section~\ref{Tw} allow us to prove the following.

\begin{theorem}\label{LatPrecise}
The number $d'$ satisfies \twostarprimesp if and only if for every decomposition $d'=dr^2$
with $d$ satisfying \twostar,
there exists an element $[w]\in \labdrdual$ of order $r$ such that
$\kdprimeperp$ is isomorphic to $\ker((w,-)\colon \labd\to\Q/\Z)$.
\end{theorem}

For a cubic fourfold $X\in\mC_{d'}$, the inclusion 
$\kdprimeperp\subset\HH^4(X,\Z)(1)$
gives an induced Hodge structure of K3 type on $\kdprimeperp$ and thus on $T_w=\ker\, (w,-)$,
yielding a point $x$ in the period domain $\mD(T_w)$.
In \cite{Voisin}, it was shown that for a smooth cubic fourfold $X$,
there are no classes in $\HH^4(X,\Z)_{\prim}\cap\HH^{2,2}(X)$ of square 2.
It follows that the class of $x$ in $\mD(T_w)/\stab[w]$ lies in the image
of the period map $\mP_w$.
As a consequence, we obtain
\begin{corollary}\label{geomresult}
A cubic fourfold $X$ is in $\mC_{d'}$ for some $d'$ satisfying \twostarprimesp if and only if
for every decomposition $d'=dr^2$ with $d$ satisfying \twostar,
there exists a polarized K3 surface $(S,L)$ of degree $d$
and an element $\alpha\in\hom(\HH^2(S,\Z)_{\prim},\Q/\Z)$ of order $r$ such that
$\kdprimeperp$ is Hodge isometric to $\ker \alpha$.
\end{corollary}
We will say that the twisted K3 surface in Corollary~\ref{geomresult}
is \emph{associated} to $X$. 

\begin{remark}
One can show that this notion of associated twisted K3 surfaces is the same 
as the one given by Huybrechts \cite{K3category}.
He relates the full cohomology $\widetilde{\HH}(S,\alpha,\Z)$ to the Hodge structure
$\widetilde{\HH}(\mA_X,\Z)$
of K3 type associated to the K3 category $\mA_X\subset\Db(X)$,
which was introduced in \cite{AddingtonThomas}.

To be precise, Huybrechts shows that a cubic $X$ is in $\mC_{d'}$
for some $d'$ satisfying \twostarprimesp
if and only if there is a twisted K3 surface $(S,\alpha)$ such that $\widetilde{\HH}(\mA_X,\Z)$
is Hodge isometric to $\widetilde{\HH}(S,\alpha,\Z)$.
When $S$ has Picard number one, it follows that ${\kdperp\subset\HH^4(X,\Z)(1)}$ is Hodge isometric
to $\ker(\alpha\colon \HH^2(S,\Z)_{\prim}\to\Q/\Z)$
(these are the transcendental parts of $\widetilde{\HH}(\mA_X,\Z)$ and $\widetilde{\HH}(S,\alpha,\Z)$).
When $\rho(S)>1$, the same holds for some lift of ${\alpha\in\br(S)}$
to $\hom(\HH^2(S,\Z)_{\prim},\Q/\Z)$.
Vice versa, one can show that a Hodge isometry
${\kdprimeperp\cong\ker(\alpha\colon \HH^2(S,\Z)_{\prim}\to\Q/\Z)}$
always extends to ${\widetilde{\HH}(\mA_X,\Z)\cong\widetilde{\HH}(S,\alpha,\Z)}$,
see Proposition~\ref{discsurj}.

In particular, this proves a weaker version of Corollary~\ref{geomresult}, 
replacing ``every decomposition'' with ``some decomposition''.
\end{remark}

We prove Theorem~\ref{LatPrecise} by comparing the quadratic forms
on $\Disc \kdprimeperp$ and $\Disc T_w$.
We distinguish the cases when the groups are cyclic and non-cyclic.
We will use the following statements, which
follow from quadratic reciprocity \cite[proof of Prop.~5.1.4]{HassettPaper}.
\begin{lemma}\label{twostarquadrec}
When $d\equiv 2\mod 6$, then $d$ satisfies \twostarsp if and only if
$-3$ is a square modulo $2d$. When $d\equiv 0\mod 6$, write $d=6t$. Then
$d$ satisfies \twostarsp if and only if $-3$ is a square modulo $4t$
and $4t$ is a square modulo 3.
\end{lemma}

\subsubsection{Proof of Theorem~\ref{LatPrecise}, cyclic case}
Assuming that $\Disc \kdprimeperp$ is cyclic,
we will show that $d'=dr^2$ with $d$ satisfying \twostarprimesp
if and only if there exists a $[w]\in\labdrdual$ 
of order $r$ such that $\kdprimeperp\cong T_w$.
The proof consists of Propositions~\ref{Cyctwostar} and \ref{Cycexists}.

\begin{proposition}\label{Cyctwostar}
Assume that $\Disc \kdprimeperp$ is cyclic. 
If there is a $[w]\in\labdrdual$ of order $r$ such that
$\kdprimeperp\cong T_w$ (so in particular, $d'=dr^2$), 
then $d$ satisfies \twostar.
\end{proposition}
\begin{proof}
First assume that $3$ does not divide $d$.
By Proposition~\ref{HassettQuadForm} and Corollary~\ref{generatorsDisc},
we have $\kdprimeperp\cong T_{w_{n,k}}$ if and only if there is an $x$ such that
\[\tfrac{x^2}{r^2d}(k^2-2nd)\equiv\tfrac{2dr^2-1}{3dr^2}\mod 2.\]
Multiplying by $3dr^2$ gives
\[3x^2(k^2-2nd)\equiv 2dr^2-1\mod 6dr^2\]
which is equivalent to
\begin{equation}\label{CycNotDivQuads}
3x^2(k^2-2nd)\equiv -1\mod 2dr^2.
\end{equation}
It follows that $-3$ is a square modulo $2d$, so by Lemma \ref{twostarquadrec},
$d$ satisfies $(**)$.

\medskip
Next we assume $3|d$.
By Proposition~\ref{HassettQuadForm} and Corollary~\ref{generatorsDisc},
we have $\kdprimeperp\cong T_{w_{n,k}}$ if and only if there is an $x$ such that
\[\tfrac{x^2}{r^2d}(k^2-2nd)\equiv \tfrac{2}{3}-\tfrac{3}{dr^2}\mod 2.\]
Writing $d=6t$ and multiplying by $dr^2$ gives
\begin{equation}\label{CycDivQuads}
x^2(k^2-12nt)\equiv 4tr^2-3 \mod 12tr^2.
\end{equation}
In particular,  $-3$ is a square modulo $4t$, and we have $4tr^2\equiv x^2k^2\mod 3$.
Since $3$ does not divide $r$, this implies that $4t$ is a square modulo 3.
It follows from Lemma~\ref{twostarquadrec} that $d$ satisfies $(**)$.
\end{proof}

Write $r = 2^sqr_0$ where $q$ consists of all prime factors of $r$ which are 1 modulo 3,
and $r_0$ consists of all odd prime factors of $r$ which are 2 modulo 3.
In particular, $dq^2$ still satisfies \twostar, and we have $\gcd(r_0,dq^2)=1$.
\begin{proposition}\label{Cycexists}
There exists an $n$ such that for $w=w_{nq^2,r_0}\in\tfrac1r\labddual$, we have
$\kdprimeperp\cong T_w$.
\end{proposition}
\begin{proof}
We first assume $3\centernot| d$. By \eqref{CycNotDivQuads} we have to show
that for some $x$ and some $n$,
\begin{equation}\label{CycNotDiv}
f_n(x):=3x^2(r_0^2-2ndq^2)+1\equiv 0\mod m
\end{equation}
where $m=2dr^2$.

 Since $dq^2$ satisfies $(**)$, the number $-3$ is a square modulo $2dq^2$.
 As $3r_0$ is invertible in $\Z/2dq^2\Z$, 
 we get $-3\equiv (3r_0x)^2\mod 2dq^2$ for some $x\in\Z$.
 This gives ${3x^2r_0^2+1\equiv 0\mod 2dq^2}$, which shows that \eqref{CycNotDiv}
 has a solution modulo $m=2dq^2$, for any $n$.
 In particular, it has solutions modulo $dq^2/2$ and modulo 4.

It follows that $(f_n/2)(x)\equiv 0$ has a solution modulo 2.
Also, $(f_n/2)'(x)=3x(r_0^2-2ndq^2)$ is always odd.
By Hensel's lemma, $(f_n/2)(x)=0$ has a solution modulo $2^l$ for any $l\geq 1$.
It follows that \eqref{CycNotDiv} has a solution modulo $2^l$ for any $l\geq 2$.

By the Chinese remainder theorem, there exists a solution $x$ for \eqref{CycNotDiv}
modulo $m=2d(2^sq)^2$.
We can assume $\gcd(x,r_0)=1$: otherwise, write $ar_0+b\cdot 2d(2^sq)^2=1$ and replace $x$ by
 ${x+b\cdot 2d(2^sq)^2(1-x) = 1+ar_0(x-1)}$.

 Now we have $\gcd(r_0^2,6x^2dq^2)=1$, so there exist $a$ and $b$ such that
 $ar_0^2+b\cdot 6x^2dq^2=1$.
 In particular, $r_0^2$ divides $3x^2\cdot -2bdq^2+1$. We see that for $n=b$,
 there is a solution to \eqref{CycNotDiv} modulo $m=r_0^2$.
 By the Chinese remainder theorem, there exists a solution modulo $2dr^2$.

 \medskip
 Next, assume $3|d$. Write $d=6t$.
 By \eqref{CycDivQuads} we have to show that for some $x$ and some $n$,
\begin{equation}\label{CycDiv}
 g_n(x):= x^2(r_0^2-12ntq^2)-4tr^2+3 \equiv 0\mod m
 \end{equation}
 where $m=12tr^2$.

 Since $dq^2$ satisfies $(**)$, first, $4tq^2$ is a square modulo 3,
 so also $4tr^2=4t(2^sqr_0)^2$ is a square modulo 3.
 Second, $-3$ is a square modulo $4tq^2$.
 Since $3$ does not divide $4tq^2$, it follows that $4tr^2-3$ is a square modulo $12tq^2$.

 Now $r_0$ is invertible in $\Z/12tq^2\Z$, which implies that
 $4tr^2-3\equiv (xr_0)^2\mod 12tq^2$ for some $x$.
 So $x^2r_0^2-4tr^2+3$ is divisible by $12tq^2$, which shows that
 \eqref{CycDiv} has a solution modulo $m=12tq^2$, for any $n$.
 In particular, there exist solutions modulo $3tq^2$
 and modulo $4$.

Like before, it follows from Hensel's lemma
that \eqref{CycDiv} has a solution modulo $2^l$ for any $l\geq 2$.

By the Chinese remainder theorem, there exists a solution $x$ for \eqref{CycDiv}
modulo $12t(2^sq)^2$.
Like before, if $\gcd(x,r_0)\neq 1$, take $a$ and $b$ such that $ar_0+b\cdot 12t(2^sq)^2 = 1$
and replace $x$ by $x+b\cdot 12t(2^sq)^2\cdot(1-x) = 1+ar_0(x-1)$.

Now we have $\gcd(r_0^2,4tx^2q^2)=1$, so we can write
$3ar_0^2 + bx^2\cdot 12tx^2q^2=3$ for some $a$ and $b$.
So for $n=b$, we find that $r_0^2$ divides $x^2\cdot-12ntq^2+3$,
hence \eqref{CycDiv} has a solution modulo $m=r_0^2$.
By the Chinese remainder theorem, it has a solution modulo $2dr^2$.
\end{proof}

\subsubsection{Proof of Theorem~\ref{LatPrecise}, non-cyclic case}
We now assume $\Disc(\kdprimeperp)\cong \Z/3\Z\times\Z/\tfrac{d'}3\Z$,
and we again show that $d'=dr^2$ with $d$ satisfying \twostarprimesp
if and only if there exists a $[w]\in\labdrdual$ of order $r$
such that $\kdprimeperp\cong T_w$.
The proof consists of Propositions~\ref{NotCyctwostar}, \ref{NotCycexists1}
and \ref{NotCycexists2}.

\begin{proposition}\label{NotCyctwostar}
Assume that $\Disc \kdprimeperp\cong \Z/(d'/3)\Z\times\Z/3\Z$.
If there is a $[w]\in\labdrdual$ of order $r$ such that
$\kdprimeperp\cong T_w$ (so in particular, $d'=dr^2$), 
then $d$ satisfies \twostar.
\end{proposition}
\begin{proof}
Consider the factor $\Z/(d'/3)\Z=\Z/(dr^2/3)\Z$.
By Proposition~\ref{HassettQuadForm} and Corollary \ref{generatorsDisc}, 
there exists an $x$ such that
$x^2\tfrac{2nd-k^2}{r^2d}$ is congruent to $\tfrac{3}{dr^2}$ modulo 2.
Multiplying both expressions with $-dr^2$ gives
\begin{equation}\label{NotCycQuads}
x^2(k^2-2nd)\equiv -3\mod 2dr^2.
\end{equation}
We see that $-3$ is a square modulo $2d$, which implies that $d$ satisfies $(**)$.
\end{proof}

Write $r = 2^sqr_0$, where $q$ consists of all prime factors of $r$ which are congruent to 1 modulo 3,
and $r_0$ consists of all other odd prime factors of $r$.
In particular, $dq^2$ still satisfies $(**)$ and $\gcd(r_0,dq^2)=1$.
Note that $3$ divides $r_0$.
\begin{proposition}\label{NotCycexists1}
Suppose $3$ does not divide $d$.
There exists an integer $n$ such that for ${w=w_{3nq^2,r_0}\in\tfrac1r\labddual}$, we have
$\kdprimeperp\cong T_w$.
\end{proposition}
\begin{proof}
By \eqref{NotCycQuads}, we need $n$ and $x$ such that
\begin{equation}\label{NotCycNotDiv}
x^2(r_0^2-6ndq^2)+3\equiv 0\mod m
\end{equation}
where $m=2dr^2$.

Since $dq^2$ satisfies \twostar, $-3$ is a square modulo $2dq^2$,
and as $r_0$ is divisible in $\Z/2dq^2\Z$, we have
$-3\equiv(r_0x)^2\mod 2dq^2$ for some $x$.
This shows that \eqref{NotCycNotDiv} has a solution modulo $2dq^2$ for any $n$.
In particular, there exist solutions modulo $dq^2/2$ and modulo $4$.

Using Hensel's lemma again, one shows that there exist solutions
modulo $2^{\ell}$ for any $\ell\geq 2$,
and by the Chinese remainder theorem, there exists a solution $x$
modulo $m=2d(2^sq)^2$.
We may assume $\gcd(x,r_0)=1$ by
writing $ar_0+b\cdot 2d(2^sq)^2=1$ and replacing $x$ by
 ${x+b\cdot 2d(2^sq)^2(1-x) = 1+ar_0(x-1)}$.
 
Now $\gcd(r_0^2/3,2x^2dq^2)=1$; take $a$ and $b$ such that
$ar_0^2/3+b\cdot 2x^2dq^2=1$.
Then $r_0^2$ divides $-6bx^2dq^2+3$,
so for $n=b$, there exists a solution to \eqref{NotCycNotDiv} modulo $r_0^2$.
By the Chinese remainder theorem,
there is a solution modulo $m=2dr^2$.

\medskip
We still need to check that for the generator $u$ of $\Z/3\Z\subset\Disc T_w$,
there exists $y\in\Z$ such that $y^2(u)^2=y^2\cdot\tfrac{d}{9}(r_0^2-6nq^2d)$
is congruent to $-2/3$ modulo 2.
Multiplying both expressions by $3/2$ gives
\[y^2\tfrac{d}{2}\tfrac{r_0^2-6nq^2d}{3}\equiv -1\mod 3.\]
Now note that $d/2\equiv 1\mod 3$, so taking $y$ such that $3$ does not divide $y$, we have
\[y^2\cdot\tfrac{d}{2}\cdot\tfrac{r_0^2-2nq^2d}{3}\equiv \tfrac{r_0^2-6nq^2d}{3}\mod 3.\]
The element on the right hand side is
\[3(r_0/3)^2-2bq^2d\equiv -b\mod 3\]
where $b$ was defined by the equation $ar_0^2/3+b\cdot 2x^2dq^2=1$.
Reducing this modulo 3, we indeed find $b\equiv 1\mod 3$.
\end{proof}

We are left with the case $3|d$.
\begin{proposition}\label{NotCycexists2}
Suppose $3$ divides $d$.
There exists an $n$ such that for $w=w_{nq^2,3r_0}\in\tfrac1r\labddual$, 
we have
$\kdprimeperp\cong T_w$.
\end{proposition}
\begin{proof}
By \eqref{NotCycQuads}, we need $n$ and $x$ such that
\[x^2((3r_0)^2-2ndq^2)+3\equiv 0\mod 2dr^2.\]
Write $d=6t$, then this is equivalent to
\begin{equation}\label{NotCycDiv}
x^2(3r_0^2-4ntq^2)+1\equiv 0\mod m
\end{equation}
where $m=4tr^2$.

As $dq^2$ satisfies \twostar, Lemma \ref{twostarquadrec} tells us that
$-3$ is a square modulo $4tq^2$. 
Since $\gcd(3r_0,4tq^2)=1$, it follows that $-3\equiv(3r_0x)^2\mod 4tq^2$
for some $x$. So we have ${3x^2r_0^2+1\equiv 0\mod 4tq^2}$,
which shows that \eqref{NotCycDiv} has a solution modulo $m=4tq^2$.

By Hensel's lemma once more, it also has a solution modulo $2^{\ell}$
for all $\ell\geq 2$, and by the Chinese remainder theorem it then has a
solution $x$ modulo $4t(2^sq)^2$.
Like before, we may assume $\gcd(x,r_0)=1$ by 
writing $ar_0+b\cdot 2d(2^sq)^2=1$ and replacing $x$ by
 ${x+b\cdot 2d(2^sq)^2(1-x) = 1+ar_0(x-1)}$.
 
Now note that $\gcd(r_0^2,4tx^2q^2)=1$ and take $a$, $b$ such that
$ar_0^2+b\cdot 4tx^2q^2=1$. Then $r_0^2$ divides $-b\cdot 4tx^2q^2+1$,
showing that for $n=b$, \eqref{NotCycDiv} has a solution modulo $m=r_0^2$.
By the Chinese remainder theorem, there exists a solution modulo $4tr^2$.

\medskip
Finally, we need to check that for the generator $u'$ of $\Z/3\Z\subset\Disc T_w$,
there exists a $y$ such that $y^2(u')^2=-y^2d/9$
is congruent to $-2/3$ modulo 2.
Multiplying by $-3/2$, we get
\[y^2d/6\equiv 1\mod 3\]
which is true whenever $3$ does not divide $y$.
\end{proof}

\section{Rational maps to \texorpdfstring{$\mC_{d'}$}{Cd'}}

For untwisted K3 surfaces, an isomorphism $\labd\cong\kdperp$ can be used to construct
a rational map $\tM_d\dashrightarrow\mC_d$.
We will generalize these maps to the situation of twisted K3 surfaces.

Throughout this section, we will assume $d'$ satisfies \twostarprimesp and fix a decomposition $d'=dr^2$ with $d$ satisfying \twostar.
Moreover, we fix $[w]\in\labdrdual$ as in Theorem \ref{LatPrecise}
and choose an isomorphism $\kdprimeperp\cong T_w=\ker\, (w,-)$.

\subsection{Construction}
The cohomology lattice $\HH^4(X,\Z)(1)$ of a cubic fourfold $X$ is isomorphic to
\[E_8(-1)^{\oplus 2}\oplus U^{\oplus 2}\oplus \Z(-1)^{\oplus 3}.\]
The isomorphism can be chosen such that the square of the hyperplane class on $X$
is mapped to $h:=(1,1,1)\in \Z(-1)^{\oplus 3}$.
We denote the orthogonal complement to $h$ by $\Gamma$,
so $\Gamma$ is isomorphic to $\HH^4(X,\Z)_{\prim}(1)$. A primitive sublattice
\[K\subset E_8(-1)^{\oplus 2}\oplus U^{\oplus 2}\oplus \Z(-1)^{\oplus 3}\]
of rank two and discriminant $d'$ containing $h$ is unique up to the action of
the stable orthogonal group
$\stO(\Gamma)=\ker(\tO(\Gamma)\to\tO(\Disc \Gamma))$.
We fix one such $K$ for each discriminant $d'$ and denote it by $K_{d'}$.
Its orthogonal complement $\kdprimeperp$ is contained in $\Gamma$.

Note that the group $\stO(\kdprimeperp)$ can be viewed as a subgroup of $\stO(\Gamma)$:
any element $f$ in $\stO(\kdprimeperp)$ can be extended to 
an orthogonal transformation $\widetilde{f}$ of the unimodular lattice
${E_8(-1)^{\oplus 2}\oplus U^{\oplus 2}\oplus \Z(-1)^{\oplus 3}}$ 
such that $\widetilde{f}|_{K_{d'}}$ is the identity.
Then restrict to $\Gamma$ to get an element of $\stO(\Gamma)$.

\medskip
On the level of the period domain, the above gives us a commutative diagram
\[\xymatrix{ \mD(T_w)\ar[d]\ar[r]^{\cong} & \mD(\kdprimeperp)\ar[d] \ar@{^{(}->}[rr] &
& \mD(\Gamma)\ar[d]\\
 \mD(T_w)/\stO(T_w) \ar[r]^{\cong} & \mD(\kdprimeperp)/\stO(\kdprimeperp) \ar[r]
 &\overline{\mC}_{d'} \ar@{^{(}->}[r] & \mD(\Gamma)/\stO(\Gamma)
}\]
where $\overline{\mC}_{d'}$ is the image of $\mD(\kdprimeperp)$ under
$\mD(\Gamma)\to\mD(\Gamma)/\stO(\Gamma)$.
Embedding the moduli space $\mC$ of smooth cubic fourfolds into $\mD(\Gamma)/\stO(\Gamma)$
via the period map,
one shows that $\overline{\mC}_{d'}$ is the closure of $\mC_{d'}\subset\mC$ in
$\mD(\Gamma)/\stO(\Gamma)$.

\begin{lemma}
The group $\stO(T_w)$ is a subgroup of $\stab[w]\subset\stO(\labd)$.
\end{lemma}
\begin{proof}
Let $g\in\stO(T_w)$.
By assumption, $g^{\dual}$ sends any $x\in T_w^{\dual}$ to $x+y$ for some $y\in T_w\subset\labd$.
In particular, this holds for $x\in\labd\subset T_w^{\dual}$,
which shows that $g^{\dual}$ preserves $\labd$.
Moreover, $g^{\dual}$ induces the identity on $\Disc \labd$,
so $g^{\dual}|_{\labd}$ is an element of $\stO(\labd)$.

Now $w\in\tfrac1r\labddual$ lies in $T_w^{\dual}$,
so we also have $g^{\dual}(w)=w+y$ for some $y\in T_w\subset\labddual$.
This implies that when acting on $\labdrdual$, the map $g^{\dual}|_{\labd}$ stabilizes $[w]$.
\end{proof}

The period map $\mP_w\colon\mdmar\times\{[w]\}\to\mD(T_w)$ induces an embedding of
\[\tildmw:= (\mdmar\times\{[w]\})/\stO(T_w)\]
into $\mD(T_w)/\stO(T_w)$. This map is algebraic, which is shown
similarly as for the embedding $\mw\hookrightarrow\mD(T_w)/\stab[w]$.
The space $\tildmw$ parametrizes tuples $(S,L,\alpha,f)$
where $(S,L,\alpha)$ is in $\mw$ and $f$ is an isomorphism from $\Disc (\ker\alpha)$ to $\Disc T_w$.
The composition
\[\tildmw\to \mD(T_w)/\stO(T_w)\to \overline{\mC}_{d'}\]
induces a rational map $\tildmw\dashrightarrow \mC_{d'}$,
which is regular on an open subset that maps surjectively 
(by Corollary~\ref{geomresult}) to $\mC_{d'}$.
Hassett showed that 
$\mD(\kdprimeperp)/\stO(\kdprimeperp) \to \overline{\mC}_{d'}$ 
generically has degree one when $d'\equiv 2\mod 6$, and degree two when $d'\equiv 0\mod 6$.
Hence, $\tildmw\dashrightarrow \mC_{d'}$ is birational
in the first case and has degree two in the second case;
see also Section~\ref{pairs}.

\medskip
The map $\gamma\colon\tildmw\dashrightarrow \mC_{d'}$ is in general not unique:
it depends on the choice of an isomorphism $T_w\cong\kdprimeperp$.
To be precise, let $\iota\colon \tO(T_w)\to\aut(\mD(T_w))$ send an isometry of $T_w$
to the induced action on the period domain.
Then $\gamma$ is unique up to 
$\iota(\tO(T_w))/\iota(\stO(T_w))$.
We can compute this group as in \cite[Lemma~3.1]{HLOY}: there is a short exact sequence
\[0\to\stO(T_w)\to\tO(T_w)\to\tO(\Disc T_w)\to 0.\]
Using $\iota(\stO(T_w))\cong \stO(T_w)$ and $\iota(\tO(T_w))\cong \tO(T_w)/\pm\id$, we find that
\[\iota(\tO(T_w))/\stO(T_w)\cong\tO(\Disc T_w)/\pm\id.\]
\begin{corollary}\label{ChoiceMap}
The map $\tildmw\dashrightarrow \mC_{d'}$ is unique up to elements of
$\tO(\Disc T_w)/\pm\id$.
\end{corollary}
When $\Disc T_w\cong\Z/d'\Z$, this group is isomorphic to $(\Z/2\Z)^{\oplus \tau(d'/2)-1}$,
where $\tau(d'/2)$ is the number of prime factors of $d'/2$.

\medskip
We have seen that there is a difference to the untwisted situation: 
the rational map to $\mC_{d'}$
can only be defined after taking a finite covering
$\pi\colon\tildmw\to \mw$.
We give an upper bound for the degree of this covering. 
\begin{corollary}\label{index}
The degree of the quotient map $\pi\colon\tildmw\to\mw$ is at most
\[I=|\tO(\Disc T_w)/\pm\id|.\]
If $\Disc T_w$ is cyclic, then $I=2^{\tau(d'/2)-1}$.
\end{corollary}
\begin{proof}
The degree of $\pi$ is the index of $\iota(\stO(T_w))\cong\stO(T_w)$ in $\iota(\stab [w])$.
This is at most the index $I$ of $\stO(T_w)$ in $\iota(\tO(T_w))$.
\end{proof}

\subsection{Example}
We consider the case $d=r=2$, so $d'=8$. 
The cubic fourfolds in $\mC_8$ are those containing a plane.
Given a generic such cubic $X$, there is
a geometric construction yielding
a twisted K3 surface $(S,\alpha)$ of degree 2 and order 2.
Kuznetsov showed that $(S,\alpha)$ is associated to $X$ in a derived category-theoretic sense \cite{Kuznetsov}; it follows from \cite{K3category}
that $(S,\alpha)$ is also Hodge-theoretically associated to $X$.

\medskip
By Lemma~\ref{orbits},
the moduli space $\tM_2^2$ has at most four connected components,
corresponding to the vectors $w_{n,k}=\tfrac12(e_1+nf_1+\tfrac{k}{2}\ell_2')$
with $n,k\in\{0,1\}$.
Now by Eichler's criterion,
$e_1$ is equivalent to $e_1+f_1+\ell_2'$ under $\stO(\Lambda_2)$,
and this is equivalent to $e_1+f_1$ modulo 
$2\Lambda_2^{\dual}$.
Thus, the components $\tM_{w_{0,0}}$ and $\tM_{w_{1,0}}$ are the same.

\medskip
The discriminant group of $K_8^{\perp}$ is cyclic, and one can choose a generator $u$ such
that $q_{K_8^{\perp}}(u)=-\tfrac{5}{8}\mod 2\Z$.
By Proposition~\ref{cyclic}, the discriminant group of $T_{w_{n,k}}$ is cyclic 
if and only if $k=1$.
By Corollary \ref{generatorsDisc}, $T_{w_{n,1}}$ is isomorphic to $K_8^{\perp}$ if and only if there exists an $x\in\Z$ such that $\tfrac{x^2(4n-1)}2\equiv -\tfrac{5}{8}\mod 2$.
For $n=0$, we have
\[\tfrac{x^2(4n-1)}2 = -\tfrac{x^2}{8}\]
which is never equivalent to $-\tfrac{5}{8}$ modulo 2.
For $n=1$, we have
\[\tfrac{x^2(4n-1)}2 = \tfrac{3x^2}{8}\]
which is equivalent to $-\tfrac{5}{8}$ modulo 2 when $x=3$.

\medskip
We see that for $v=w_{1,1}$, there exists a rational map 
$\tildmw\dashrightarrow \mC_{d'}$ as above.
Since $d'/2=4$ has only one prime factor, 
Corollary \ref{ChoiceMap} tells us that there is a unique choice for
the rational map $\widetilde{\tM}_{w_{1,1}}\dashrightarrow\mC_8$.
Moreover, it follows from Corollary~\ref{index} that
$\pi\colon \widetilde{\tM}_{w_{1,1}}\to \tM_{w_{1,1}}$ is an isomorphism.
Hence, we obtain a rational map
\[\tM_{w_{1,1}}\dashrightarrow \mC_8\]
which gives an inverse to the geometric construction of associated twisted K3 surfaces 
over the locus where $\rho(S)=1$.

\begin{remark}
The three types of Brauer classes occurring in this example have been studied before 
by Van Geemen \cite{Geemen} (see also \cite[\S2]{BrauerOdd}).
He relates the twisted K3 surfaces in the components $\tM_{w_{0,0}}$ and $\tM_{w_{0,1}}$
to certain double covers of $\P^2\times\P^2$ and to complete intersections of three quartics
in $\P^4$, respectively.
\end{remark}

\begin{remark}
In general, the component $\mw\subset\tM_d^r$ for which a rational map $\tildmw\dashrightarrow \mC_{d'}$
exists is not unique, because the class $[w]\in\labdrdual$ satisfying $T_w\cong\kdprimeperp$
is not unique modulo $\stO(\labd)$. 
We work out an example.

Let $d=14$ and $r=7$, so $\Disc \kdprimeperp$ is cyclic.
Since $r$ divides $d$, \cite[Thm.~9]{BrauerOdd} tells us that for $[w]\in\labdrdual$ of order $r$,
there is only one isomorphism class of lattices $T_w$
with cyclic discriminant group.
By Theorem~\ref{LatPrecise}, these $T_w$ are isomorphic to $\kdprimeperp$.

Consider $w_{0,1}=\tfrac17(e_1+\ell_{14}'/14)$ and 
$w_{1,3}=\tfrac17(e_1+f_1+3\ell_{14}'/14)$.
By Proposition~\ref{cyclic}, $\Disc T_{w_{0,1}}$ and $\Disc T_{w_{1,3}}$
are both cyclic. By the above, we have $T_{w_{0,1}}\cong T_{w_{1,3}}\cong K_{14\cdot 7^2}^{\perp}$.
We show that $[w_{0,1}]\not\equiv [w_{1,3}]$ in $\labdrdual/\stO(\labd)$.

Namely, suppose $[w_{1,3}]$ lies in the orbit 
$\stO(\Lambda_{14})\cdot [w_{0,1}]\subset \Lambda_{14,7}^{\dual}$.
Then there exists $z\in\Lambda_{14}^{\dual}$
such that $f_7(w_{0,1})=w_{1,3}+z$ for some $f\in\stO(\Lambda_{14})$,
that is, $f(14\cdot 7w_{0,1}) = 14\cdot 7( w_{1,3}+z)$.
Write $z=z_0+\tfrac{t}{14}\ell_{14}'$ for some $z_0\in\epart\oplus U_1\oplus U_2$ and $t\in\Z$, so
\[14\cdot 7(w_{1,3}+z) = 14(e_1+f_1)+14\cdot 7z_0+(3+7t)\ell_{14}'.\]
The square of the right hand side 
should be equal to $(14\cdot 7w_{0,1})^2=-14$. This gives
\[-14=2\cdot 14^2+14^3(e_1+f_1,z_0)+(14\cdot 7)^2(z_0)^2-14(9+6\cdot 7t+(7t)^2)\]
which simplifies to
\[8=2\cdot 14+14^2(e_1+f_1,z_0)+14\cdot 7^2(z_0)^2-(6\cdot 7t+(7t)^2).\]
Reducing modulo 7, one sees that this is not possible.
\end{remark}

\subsection{Pairs of associated twisted K3 surfaces}\label{pairs}
In \cite{Involution}, we studied the covering involution 
of Hassett's rational map $\tM_d\dashrightarrow\mC_d$ in the case this has degree two.
We showed that if $(S,L)\in\tM_d$ is mapped to $(S^{\tau},L^{\tau})$
under this involution, then $S^{\tau}$ is isomorphic to a moduli space
of stable sheaves on $S$ with Mukai vector $(3,L,d/6)$.
In this section, we discuss the analogous twisted situation.

We denote the bounded derived category 
of $\alpha$-twisted coherent sheaves on $S$ by $\Db(S,\alpha)$.
When $\alpha$ lies in $\hom(\HH^2(S,\Z)_{\prim},\Z/r\Z)\hookrightarrow\HH^2(S,\Q)$,
then by $\alpha$-twisted sheaves we mean $\exp(\alpha^{0,2})$-twisted sheaves
and $\HH^2(S,\alpha,\Z)$ means $\HH^2(S,\exp(\alpha^{0,2}),\Z)$.

\medskip
Assume that $3$ divides $d'=dr^2$. Hassett showed (see also \cite{Involution})
that the map ${\mD(\kdprimeperp)/\stO(\kdprimeperp)\to\overline{\mC}_d}$
is a composition $\nu\circ f$, where $\nu$ is the normalization of $\overline{\mC}_d$
and $f$ is generically of degree two, induced by an element in $\tO(\kdperp)$
of order two. 
The corresponding element $g\in\tO(T_w)$ induces a covering involution 
\[\tau\colon \mD(\kdprimeperp)/\stO(\kdprimeperp)\to \mD(\kdprimeperp)/\stO(\kdprimeperp)\]
that preserves $\tildmw$. 
We claim that $g$ extends to an orthogonal transformation of $\widetilde{\Lambda}$.
This follows from \cite[Cor.~1.5.2]{Nikulin} and the following statement.
We embed $T_w\subset\labd$ primitively into $\widetilde{\Lambda}$ 
using the map $\exp(w)$, as in Section \ref{Periodmaps}.

\begin{proposition}\label{discsurj}
Let $S_w:=T_w^{\perp}\subset\widetilde{\Lambda}$. The map
$\tO(S_w)\to\tO(\Disc S_w)$ is surjective.
\end{proposition}
\begin{proof}
The lattice $S_w$ has rank three. When $\Disc T_w\cong \Disc S_w$ is cyclic,
the statement follows from \cite[Thm.~1.14.2]{Nikulin}.
When $\Disc T_w$ is $\Z/(d'/3)\Z\times\Z/3\Z$, it follows from
Corollary~VIII.7.3 in \cite{MiMo}.
\end{proof}

This implies that when $\tau$ maps $(S,L,\alpha,f)\in\tildmw$
to $(S',L',\alpha',f')$,
then there is a Hodge isometry
\[\widetilde{\HH}(S,\alpha,\Z)\cong \widetilde{\HH}(S',\alpha',\Z).\]
This map might not preserve the orientation of the four positive directions.
However, by \cite[Lemma~2.3]{K3category}, there exists an orientation reversing Hodge
isometry in $\tO(\widetilde{\HH}(S,\alpha,\Z))$.
By composing with it, we see that there exists a Hodge isometry 
$\widetilde{\HH}(S,\alpha,\Z)\cong \widetilde{\HH}(S',\alpha',\Z)$ which is orientation
preserving.
By \cite{CaldararusConj}, there is an equivalence 
\[\Phi\colon \Db(S,\alpha)\to \Db(S',\alpha')\]
which is of Fourier--Mukai type:
there exists $\mE\in \Db(S\times S',\alpha^{-1}\boxtimes\alpha')$ and an isomorphism of
functors $\Phi\cong\Phi_{\mE}$.
Let $\Phi_{\mE}^H\colon \widetilde{\HH}(S,\alpha,\Z)\to \widetilde{\HH}(S',\alpha',\Z)$
be the associated cohomological Fourier--Mukai transform.
Then $S'$ is a moduli space of complexes of $\alpha$-twisted sheaves on 
$S$ with Mukai vector 
\[v=(\Phi_{\mE}^H)^{-1}(v(k(x)))=(\Phi_{\mE}^H)^{-1}(0,0,1),\] 
where $x$ is any closed point in $S'$.
It is a coarse moduli space: the universal family on $S\times S'$ exists as an 
$\alpha^{-1}\boxtimes\alpha'$-twisted sheaf, which is an untwisted sheaf 
if and only if $\alpha'$ is trivial.

\medskip
In fact, one can show that $S'$ is isomorphic to a moduli space of stable
$\alpha$-twisted \emph{sheaves} on $S$. 
Namely, by \cite{Yoshioka} (see also \cite{CaldararusConj}),
there exists a (coarse) moduli space $M(v)$ of stable $\alpha$-twisted sheaves
on $S$ with Mukai vector $v$.
By precomposing $\Phi_{\mE}$ with autoequivalences of $\Db(S,\alpha)$,
we may assume $M(v)$ is non-empty \cite[\S2]{CaldararusConj}.
Hence, as $(v)^2=0$, the space $M(v)$ is a K3 surface.

For some $B$-field $\beta\in\HH^2(M(v),\Q)$, there exists a universal family
$\mE_v$ on $S\times M(v)$ which is an $\alpha^{-1}\boxtimes\beta$-twisted sheaf.
It induces an equivalence 
${\Phi_{\mE_v}\colon \Db(S,\alpha)\to \Db(M(v),\beta)}$ whose associated
cohomological Fourier--Mukai transform $\Phi_{\mE_v}^H$ sends 
$v$ to ${(0,0,1)\in\widetilde{\HH}(M(v),\beta,\Z)}$.
The composition
\[\Phi_{\mE_v}^H\circ (\Phi_{\mE}^H)^{-1}\colon 
\widetilde{\HH}(S',\alpha',\Z)\to \widetilde{\HH}(M(v),\beta,\Z)\]
is a Hodge isometry that sends $(0,0,1)$ to $(0,0,1)$
and is orientation preserving, since both $\Phi_{\mE_v}^H$ 
and $\Phi_{\mE}^H$ are (for $\Phi_{\mE_v}^H$, see \cite{HuyStelEquivTwisted}).
It follows from \cite[\S2]{CaldararusConj} that $S'$ is isomorphic to $M(v)$.


\end{document}